\documentclass[11pt,reqno]{amsart}
\usepackage{amsmath}
\usepackage{amssymb}
\usepackage{amstext}
\usepackage{amsfonts}
\usepackage{amsthm}
\usepackage{ifthen}
\usepackage{xspace}
\usepackage{comment}
\usepackage{graphicx}
\usepackage{mathrsfs}
\usepackage{enumerate}
\usepackage{breakcites}

\numberwithin{equation}{section}  

\newtheorem{punkt}{}[section]

\theoremstyle{plain}

\newtheorem{lemma}[punkt]{Lemma}
\newtheorem{proposition}[punkt]{Proposition}
\newtheorem{theorem}[punkt]{Theorem}

\theoremstyle{definition}

\theoremstyle{plain}
\newtheorem*{corollary*}{Corollary}
\newtheorem*{lemma*}{Lemma}
\newtheorem*{proposition*}{Proposition}
\newtheorem*{theorem*}{Theorem}

\theoremstyle{definition}
\newtheorem*{remark*}{Remark}
\newtheorem*{remarks*}{Remarks}
\newtheorem*{example*}{Example}
\newtheorem*{examples*}{Examples}
\newtheorem*{definition*}{Definition}
\newtheorem*{conjecture*}{Conjecture}
\newtheorem*{assumption*}{Assumption}
\newtheorem*{assumptions*}{Assumptions}
\newtheorem*{construction*}{Construction}

\def\mynat{\mathbb{N}}

\def\mym{\mathcal{M}}

\def\myo{\mathcal{O}}

\def\ee{\mathbb{E}}

\def\re{\qopname\relax{no}{Re}\,}

\def\ai{\qopname\relax{no}{Ai}\nolimits}

\def\eg{e.g.\@\xspace}
\def\ie{i.e.\@\xspace}
\def\iid{i.i.d.\@\xspace}

\def\wrt{w.r.t.\@\xspace}

\begin{document}

\title{Characteristic Polynomials \\[+5pt] of Sample Covariance Matrices}

\author{H. K\"osters}
\address{Holger K\"osters, Fakult\"at f\"ur Mathematik, Universit\"at Bielefeld,
Postfach 100131, 33501 Bielefeld, Germany}
\email{hkoesters@math.uni-bielefeld.de}


\begin{abstract}
We investigate the second-order correlation function 
of the characteristic polynomial of a sample covariance matrix. 
Starting from an explicit formula for a generating function, 
we re-obtain several well-known kernels from random matrix theory.
\end{abstract}

\maketitle

\markboth{H. K\"osters}{Characteristic Polynomials of Sample Covariance Matrices}

\section{Introduction}

Characteristic polynomials of random matrices have found
considerable attention in recent years, one reason being that
their correlations seem to reflect the correlations of the eigenvalues
\cite{BH1,BH2,FS3,BDS,SF3,AF3,Va,BS,GK,Ko1,Ko2}.

In this article we investigate the second-order correlation function
of the characteristic polynomial of a sample covariance matrix.

\medskip

\textbf{Complex Sample Covariance Matrices.}
Let $Q$ be a distribution on the real line  
with expectation $0$, variance $1/2$ and finite fourth moment $b$,
and for given $n,m \in \mathbb{N}$ with $n \geq m$,
let $X := X(n,m) := (X_{ij})_{i=1,\hdots,n;j=1,\hdots,m}$
denote the $n \times m$ matrix whose entries $X_{ij}$ 
are \iid complex random variables 
whose real and imaginary parts are independent,
each with distribution $Q$.
Let $X^* = X^*(n,m)$ denote the conjugate transpose of $Z$.
Then the Hermitian $m \times m$ matrix $Z := Z(n,m) := (1/n) \, X(n,m)^* X(n,m)$
is called the (complex) \emph{sample covariance matrix} associated with the distribution $Q$. 
For $\mu,\nu \in \mathbb{C}$, let
$$
f(n,m;\mu,\nu) := \ee \left( \det(X(n,m)^*X(n,m)-\mu) \det(X(n,m)^*X(n,m)-\nu) \right)
$$
denote the second-order correlation function
of the characteristic polynomial 
of the ``unrescaled'' sample covariance matrix.

\medskip

\textbf{Real Sample Covariance Matrices.}
Let $Q$ be a distribution on the real line
with ex\-pectation $0$, variance $1$ and finite fourth moment $b$,
and for given $n,m \in \mathbb{N}$ with $n \geq m$,
let $X := X(n,m) := (X_{ij})_{i=1,\hdots,n;j=1,\hdots,m}$
denote the $n \times m$ matrix whose entries $X_{ij}$ 
are \iid real random variables with distribution $Q$.
Let $X^T = X^T(n,m)$ denote the transpose of $X$.
Then the symmetric $m \times m$ matrix $Z := Z(n,m) := (1/n) \, X(n,m)^T X(n,m)$
is called the (real) \emph{sample covariance matrix} associated with the dis\-tribution $Q$. 
For $\mu,\nu \in \mathbb{C}$, let
$$
f(n,m;\mu,\nu) := \ee \left( \det(X(n,m)^TX(n,m)-\mu) \det(X(n,m)^TX(n,m)-\nu) \right)
$$
denote the second-order correlation function
of the characteristic polynomial 
of the ``unrescaled'' sample covariance matrix.

\medskip

In the special case where $Q$ is the Gaussian distribution,
the random matrix $Z$ is also called a (complex or real) \emph{Wishart matrix},
since its distribution is the (complex or real) \emph{Wishart distribution}
(see \textsc{Anderson} \cite{An} or \textsc{Muirhead} \cite{Mu}).

It will always be clear from the context whether 
we are considering the complex or real case. More precisely,
we are interested in the asymptotic behavior of the~values
$f(n_N,m_N;\mu_N,\nu_N)$ as $N \to \infty$,
where $n_N := N$, $m_N := N - \alpha$ for some fixed natural number $\alpha$,
and the parameters $\mu_N$ and $\nu_N$ are finally used 
to ``zoom in'' at certain interesting regions of the spectrum 
of the (unrescaled) sample covariance matrix,
%
%
namely the bulk of the spectrum,
the soft edge of the spectrum,
and the hard edge of the spectrum.
(See Section~3 for details.)

In the complex setting, we will recover,
in all the regions previously mentioned,
the well-known kernels for the ``correlation functions'' of the eigenvalues,
namely the sine kernel, the Airy kernel and the Bessel kernel.
(See Section~3 for details.)
Thus, our results indicate that these kernels are \emph{universal}
in that they arise in connection with quite general sample covariance matrices 
(as~described at the beginning of this section),
albeit at the level of the characteristic polynomial.
It is conjectured that this universality also holds
at the level of the eigenvalue themselves,
but so far this conjecture has only been proven
for a restricted class of sample covariance matrices
in the bulk of the~spectrum
(see \textsc{Ben Arous} and \textsc{P\'ech\'e} \cite{BP})
and under stronger moment conditions than ours
for the edges of the~spectrum 
(see \textsc{Soshnikov} \cite{So} and \textsc{Tao} and \textsc{Vu} \cite{TV}, respectively).

Similar results will also be obtained in the real setting.
However, the results for the correlation functions
of the characteristic polynomial are somewhat more different
from the results for the eigenvalues here.

To obtain our results, we follow the strategy
proposed by \textsc{G\"otze} and \textsc{K\"osters} \cite{GK}
in the context of Wigner matrices.
First, we obtain an explicit expression
for an exponential-type generating function
of the second-order correlation function
of the characteristic polynomial.
Second, we recover the well-known kernels 
from random matrix theory by asymptotic analysis.

\bigskip

\section{Generating Functions}

In this section we derive the generating function
of the second-order correlation function of the characteristic polynomial.
Since the derivation is very similar for complex and real sample covariance matrices,
we present the details only for the complex setting and restrict ourselves
to a few comments for the real setting.

The following well-known representation will prove useful:

\begin{lemma}
\label{chiral}
Let $X \in \mym(n \times m, \mathbb{C})$. Then, for $\lambda \ne 0$,
$$
  \det \left( \begin{array}{cc} \lambda I_n & X \\ X^* & \lambda I_m \end{array} \right)
= \det \left( \lambda^2 I_m - X^* X \right) \lambda^{n-m}
= \det \left( \lambda^2 I_n - X X^* \right) \lambda^{m-n} \,.
$$
\end{lemma}

\begin{proof}
This is an immediate consequence of the matrix factorizations
\begin{align*}
   \left( \begin{array}{cc} A & B \\ C & D \end{array} \right)
&= \left( \begin{array}{cc} I_n & 0 \\ CA^{-1} & D - CA^{-1}B \end{array} \right) \left( \begin{array}{cc} A & B \\ 0 & I_m \end{array} \right) \,, 
\end{align*}
\begin{align*}
   \left( \begin{array}{cc} A & B \\ C & D \end{array} \right)
&= \left( \begin{array}{cc} A - BD^{-1}C & BD^{-1} \\ 0 & I_m \end{array} \right) \left( \begin{array}{cc} I_n & 0 \\ C & D \end{array} \right) \,, 
\end{align*}
where 
$A \in \mym(n \times n, \mathbb{C})$,
$B \in \mym(n \times m, \mathbb{C})$,
$C \in \mym(m \times n, \mathbb{C})$,
$D \in \mym(m \times m, \mathbb{C})$,
and $D$ and $A$ are invertible.
\end{proof}

\pagebreak[2]

A matrix of the form $\tbinom{\ 0 \ \ X}{X^* \ \, 0 \,}$,
where $X$ is a (complex) random matrix with \iid entries,
is also called a (complex) \emph{chiral matrix}.
Thus, Lemma \ref{chiral} establishes a~connection between
the characteristic polynomial of a sample covariance matrix
and that of a chiral matrix.
Chiral matrices seem more convenient for our purposes,
as they give rise to neater row and column expansions.
 
To illustrate our approach, we start with the first moment 
of the characteristic polynomial, which has already been determined
by \textsc{Forrester} and \textsc{Gamburd} \cite{FG}
by means of a combinatorial argument.
Fix $\lambda \in \mathbb{C}$, and put 
$$
f(n,m) := \ee \det \left( \begin{array}{cc} \lambda I_n & X(n,m) \\ X(n,m)^* & \lambda I_m \end{array} \right) \,,
$$
where $X := X(n,m)$ is a matrix of size $n \times m$ with \iid matrix entries
satisfying our standing moment conditions.
Note that this definition is meaningful even if $n = 0$ or $m = 0$;
we then have $f(0,m) = \lambda^m$ and $f(n,0) = \lambda^n$.
(In particular, the~determinant of the $0 \times 0$ matrix
is defined to be $1$.)

Let us derive a recursive equation for $f(n,m)$.
Suppose that $m > 0$.
Then, doing a~row~and~column~expansion about the last row and the last column,
we~have
\begin{align}
\det \left( \begin{array}{cc} \lambda I_n & X \\ X^* & \lambda I_m \end{array} \right)
&= \lambda \det \left( \begin{array}{cc} \lambda I_{n} & X^{[\,\cdot\,:m]} \\ (X^{[\,\cdot\,:m]})^* & \lambda I_{m-1} \end{array} \right) \nonumber \\
&+ \sum_{i,j=1}^{n} (-1)^{i+j-1} X_{i,m} \overline{X}_{j,m} \det \left( \begin{array}{cc} \lambda I_{n}^{[i:j]} & X^{[i:m]} \\ (X^{[j:m]})^* & \lambda I_{m-1} \end{array} \right) \,,
\label{double-expansion}
\end{align}
where an upper index $[i:j]$ indicates that the $i$th row and the $j$th column
of the corresponding matrix are deleted. 
Taking expectations, using independence and noting that
$$
\ee X_{i,m} \overline{X}_{j,m} = \delta_{ij}
$$
(as follows from our standing moment conditions), we obtain
\begin{align}
\label{fnm-s1}
f(n,m) = \lambda f(n,m-1) - n f(n-1,m-1) \,.
\end{align}
Observe that this is meaningful even if $n = 0$, since the second term
vanishes in this case. Hence, recalling the identity
\begin{align}
\label{fnm-s0}
f(n,0) = \lambda^n \,,
\end{align}
we see that the values $f(n,m)$ are completely determined
by recursion over $m$.

For arbitrary $\alpha \in \mathbb{C}$, the Laguerre polynomials
$L_n^{(\alpha)}(x)$ are defined by
$$
L_n^{(\alpha)}(x) := \sum_{\nu=0}^{n} \binom{n+\alpha}{n-\nu} \frac{(-x)^v}{v!} 
$$
(see \eg Equation (5.1.6) in \textsc{Szeg\"o} \cite{Sz}).
It is well-known that the Laguerre poly\-nomials satisfy the relations
\begin{align}
\label{lna-1}
m L_m^{(n-m)}(x) = - x L_{m-1}^{(n-m+1)}(x) + n L_{m-1}^{(n-m)}(x) \,,
\end{align}
\begin{align}
\label{lna-0}
L_0^{(n-m)}(x) = 1 \,,
\end{align}
(see \eg Equations (5.1.10) and (5.1.14) in \textsc{Szeg\"o} \cite{Sz}),
through which they are completely determined by recursion over $m$.

\pagebreak[2]

Comparing \eqref{fnm-s1}, \eqref{fnm-s0} with \eqref{lna-1}, \eqref{lna-0},
it is easy to see (by induction on $m$) that for any $n,m \geq 0$,
\begin{align}
\label{fnm}
f(n,m) = (-1)^m \, m! \, \lambda^{n-m} \, L^{(n-m)}_{m}(\lambda^2) \,.
\end{align}
%
%
Hence, using Lemma \ref{chiral}, we find that the first moment
of the characteristic poly\-nomial of a (complex) sample covariance matrix satisfies
\begin{align}
\label{firstorder}
  \ee \det \left( \lambda I_m - X^* X \right)
= (-1)^m \, m! \, L^{n-m}_{m}(\lambda) \,.
\end{align}
In fact, this was already proved by \textsc{Forrester} and \textsc{Gamburd} \cite{FG}
by means of a~combinatorial argument.

\begin{remark*}
Equation (\ref{firstorder}) remains true for real sample covariance matrices
(under the respective moment conditions).
\end{remark*}

\begin{remark*}
We have obtained a recursive equation over $m$ for the moments $f(n,m)$
Of course, for symmetry reasons, it is clear that it is equally possible 
to derive a recursive equation over $n$. Indeed, starting with 
a row and column expansion about the first row and the first column,
we obtain
$$
f(n,m) = \lambda f(n-1,m) - m f(n-1,m-1) \,.
$$
instead of \eqref{fnm-s1}, together with the initial condition
$$
f(0,m) = \lambda^m \,.
$$
Since the Laguerre polynomials satisfy the relation
$$
L_m^{(n-m)}(x) = L_m^{(n-1-m)}(x) + L_{m-1}^{(n-1-m+1)}(x) \,,
$$
$$
L_m^{(-m)}(x) = (-x)^m / m! \,,
$$
(see \eg Equations (5.1.13) and (5.2.1) in \textsc{Szeg\"o} \cite{Sz}),
they can also be calculated by~recursion over $n$,
and (\ref{fnm}) follows (by induction on $n$).
\end{remark*}

Let us now turn to the second moment of the characteristic polynomial.
Similarly as above, we first consider 
$$
f(n,m) := \ee \left( \det \left( \begin{array}{cc} \mu I_n & X \\ X^* & \mu I_m \end{array} \right) \det \left( \begin{array}{cc} \nu I_n & X \\ X^* & \nu I_m \end{array} \right) \right) \,.
$$
A similar expansion of the determinants as in \eqref{double-expansion} yields
\begin{align*}
&\mskip24mu
\det \left( \begin{array}{cc} \mu I_n & X \\ X^* & \mu I_m \end{array} \right)
\det \left( \begin{array}{cc} \nu I_n & X \\ X^* & \nu I_m \end{array} \right) \\
&=
\left( \mu \det \left( \begin{array}{cc} \mu I_{n} & X^{[\,\cdot\,:m]} \\ (X^{[\,\cdot\,:m]})^* & \mu I_{m-1} \end{array} \right) + \sum_{i,j=1}^{n} (-1)^{i+j-1} X_{i,m} \overline{X}_{j,m} \det \left( \begin{array}{cc} \mu I_{n}^{[i:j]} & X^{[i:m]} \\ (X^{[j:m]})^* & \mu I_{m-1} \end{array} \right) \right) \\
&\mskip24mu \,\cdot\,
\left( \nu \det \left( \begin{array}{cc} \nu I_{n} & X^{[\,\cdot\,:m]} \\ (X^{[\,\cdot\,:m]})^* & \nu I_{m-1} \end{array} \right) + \sum_{k,l=1}^{n} (-1)^{k+l-1} X_{k,m} \overline{X}_{l,m} \det \left( \begin{array}{cc} \nu I_{n}^{[k:l]} & X^{[k:m]} \\ (X^{[l:m]})^* & \nu I_{m-1} \end{array} \right) \right) \,.
\end{align*}

\def\sdot{\,\cdot\,}

\noindent{}Note that due to our standing moment assumptions, we have
$$
\ee X_{ij} = 0 \,, \qquad
\ee X_{ij}^2 = 0 \,, \qquad
\ee |X_{ij}|^2 = 1 \,, \qquad
\ee |X_{ij}|^4 = 2b^2 + \tfrac12 \,,
$$
Hence, expanding the product and taking expectations, we obtain
\begin{align*}
f(n,m) 
&= 
\mu \nu \ee 
\det \left( \begin{array}{cc} \mu I_{n} & X^{[\sdot:m]} \\ (X^{[\sdot:m]})^* & \mu I_{m-1} \end{array} \right)
\det \left( \begin{array}{cc} \nu I_{n} & X^{[\sdot:m]} \\ (X^{[\sdot:m]})^* & \nu I_{m-1} \end{array} \right)
\\ &\mskip24mu \,-\,
\mu \sum_{k=1}^{n} \ee |X_{k,m}|^2 
\det \left( \begin{array}{cc} \mu I_{n} & X^{[\sdot:m]} \\ (X^{[\sdot:m]})^* & \mu I_{m-1} \end{array} \right)
\det \left( \begin{array}{cc} \nu I_{n-1} & X^{[k:m]} \\ (X^{[k:m]})^* & \nu I_{m-1} \end{array} \right)
\\ &\mskip24mu \,-\,
\nu \sum_{i=1}^{n} \ee |X_{i,m}|^2 
\det \left( \begin{array}{cc} \mu I_{n-1} & X^{[i:m]} \\ (X^{[i:m]})^* & \mu I_{m-1} \end{array} \right)
\det \left( \begin{array}{cc} \nu I_{n} & X^{[\sdot:m]} \\ (X^{[\sdot:m]})^* & \nu I_{m-1} \end{array} \right)
\\ &\mskip24mu \,+\,
\sum_{i} \ee |X_{i,m}|^4 
\det \left( \begin{array}{cc} \mu I_{n-1} & X^{[i:m]} \\ (X^{[i:m]})^* & \mu I_{m-1} \end{array} \right)
\det \left( \begin{array}{cc} \nu I_{n-1} & X^{[i:m]} \\ (X^{[i:m]})^* & \nu I_{m-1} \end{array} \right)
\\ &\mskip24mu \,+\,
\sum_{i \ne k} \ee |X_{i,m}|^2 |X_{k,m}|^2
\det \left( \begin{array}{cc} \mu I_{n-1} & X^{[i:m]} \\ (X^{[i:m]})^* & \mu I_{m-1} \end{array} \right)
\det \left( \begin{array}{cc} \nu I_{n-1} & X^{[k:m]} \\ (X^{[k:m]})^* & \nu I_{m-1} \end{array} \right)
\\ &\mskip24mu \,+\,
\sum_{i \ne j} \ee |X_{i,m}|^2 |X_{j,m}|^2
\det \left( \begin{array}{cc} \mu I_{n}^{[i:j]} & X^{[i:m]} \\ (X^{[j:m]})^* & \mu I_{m-1} \end{array} \right)
\det \left( \begin{array}{cc} \nu I_{n}^{[j:i]} & X^{[j:m]} \\ (X^{[i:m]})^* & \nu I_{m-1} \end{array} \right)\,.
\end{align*}
To shorten notation, let us introduce the auxiliary functions
$$
f(n,m,\alpha,\sdot) := \ee \left( \det \left( \begin{array}{cc} \mu I_n^{[\alpha_1:\alpha_2]} & X^{[\alpha_1:\sdot]} \\ (X^{[\alpha_2:\sdot]})^* & \mu I_m \end{array} \right) \det \left( \begin{array}{cc} \nu I_n^{[\alpha_3:\alpha_4]} & X^{[\alpha_3:\sdot]} \\ (X^{[\alpha_4:\sdot]})^* & \nu I_m \end{array} \right) \right) \,,
$$
where $\alpha = (\alpha_1,\alpha_2,\alpha_3,\alpha_4) \in \{ \sdot,1,2 \}^4$
(a dot representing no deletion), \linebreak as well as the abbreviations
$
\sdot := (\sdot,\sdot,\sdot,\sdot),\
01 := (\sdot,\sdot,1,1),\
10 := (1,1,\sdot,\sdot),\linebreak[2]
11_A := (1,1,2,2),\
11_B := (1,2,1,2),\
11_C := (1,2,2,1).
$
(No~other~values of $\alpha$ will be needed. Moreover,
the value $11_B$ will be needed for the real setting only.)
\linebreak[2] Note that $f(n,m,\sdot,\sdot) = f(n,m)$.
We then have
\begin{align}
   f(n,m) 
&= \mu \nu f(n,m-1) \nonumber\\
&- \mu n f(n,m-1,01,\sdot) \nonumber\\
&- \nu n f(n,m-1,10,\sdot) \nonumber\\
&+ (2b+\tfrac12) n f(n-1,m-1) \nonumber\\
&+ n (n-1) f(n,m-1,11_A,\sdot) \nonumber\\
&+ n (n-1) f(n,m-1,11_C,\sdot) \,.
\label{fnm-x}
\end{align}

Note that the preceding formula was obtained
by starting with an expansion ``in the $m$-dimension''.
Of course, it is possible to derive an analogous formula
by starting with an expansion ``in the $n$-dimension'';
we then have
\begin{align}
   f(n,m) 
&= \mu \nu f(n-1,m) \nonumber\\
&- \mu m f(n-1,m,\sdot,01) \nonumber\\
&- \nu m f(n-1,m,\sdot,10) \nonumber\\
&+ (2b+\tfrac12) m f(n-1,m-1) \nonumber\\
&+ m (m-1) f(n-1,m,\sdot,11_A) \nonumber\\
&+ m (m-1) f(n-1,m,\sdot,11_C) \,,
\label{fnm-y}
\end{align}
where $f(n,m,\sdot,\alpha)$ is defined by
$$
f(n,m,\sdot,\alpha) := \ee \left( \det \left( \begin{array}{cc} \mu I_n & X^{[\sdot:\alpha_1]} \\ (X^{[\sdot:\alpha_2]})^* & \mu I_m^{[\alpha_2:\alpha_1]} \end{array} \right) \det \left( \begin{array}{cc} \nu I_n & X^{[\sdot:\alpha_3]} \\ (X^{[\sdot:\alpha_4]})^* & \nu I_m^{[\alpha_4:\alpha_3]} \end{array} \right) \right) \,,
$$
and the possible values $\sdot,01,10,11_A,11_B,11_C$ for $\alpha$ 
are the same as before.

Here is a list of recursive relations which can be derived 
using the above arguments:
\begin{align}
   f(n,m,\sdot,\sdot)
&= \mu \nu f(n,m-1,\sdot,\sdot) \nonumber\\
&- \mu n f(n,m-1,01,\sdot) \nonumber\\
&- \nu n f(n,m-1,10,\sdot) \nonumber\\
&+ (2b+\tfrac12) n f(n-1,m-1,\sdot,\sdot) \nonumber\\
&+ n(n-1) f(n,m-1,11^A,\sdot) \nonumber\\
&+ n(n-1) f(n,m-1,11^C,\sdot) & (n \geq 0, m \geq 1) \,,
\label{fnm-a}
\\[+3pt]
f(n,m,\sdot,01) &= (+\mu) f(n,m-1,\sdot,\sdot) \nonumber\\ &- n f(n,m-1,10,\sdot) & (n \geq 0, m \geq 1) \,, 
\label{fnm-b}
\\[+3pt]
f(n,m,\sdot,10) &= (+\nu) f(n,m-1,\sdot,\sdot) \nonumber\\ &- n f(n,m-1,01,\sdot) & (n \geq 0, m \geq 1) \,, 
\label{fnm-c}
\allowdisplaybreaks
\\[+3pt]
   f(n,m,\sdot,11^A)
&= \mu \nu f(n,m-2,\sdot,\sdot) \nonumber\\
&- \mu n f(n,m-2,01,\sdot) \nonumber\\
&- \nu n f(n,m-2,10,\sdot) \nonumber\\
&+ n f(n-1,m-2,\sdot,\sdot) \nonumber\\
&+ n(n-1) f(n,m-2,11^A,\sdot) & (n \geq 0, m \geq 2) \,,
\label{fnm-d}
\\[+3pt]
   f(n,m,\sdot,11^C)
&= n f(n-1,m-2,\sdot,\sdot) \nonumber\\
&+ n(n-1) f(n,m-2,11^C,\sdot) & (n \geq 0, m \geq 2) \,,
\label{fnm-e}
\allowdisplaybreaks
\\[+3pt]
   f(n,m,\sdot,\sdot)
&= \mu \nu f(n-1,m,\sdot,\sdot) \nonumber\\
&- \mu m f(n-1,m,\sdot,01) \nonumber\\
&- \nu m f(n-1,m,\sdot,10) \nonumber\\
&+ (2b+\tfrac12) m f(n-1,m-1,\sdot,\sdot) \nonumber\\
&+ m(m-1) f(n-1,m,\sdot,11^A) \nonumber\\
&+ m(m-1) f(n-1,m,\sdot,11^C) & (n \geq 1, m \geq 0) \,, 
\label{fnm-f}
\\[+3pt]
f(n,m,01,\sdot) &= (+\mu) f(n-1,m,\sdot,\sdot) \nonumber\\ &- m f(n-1,m,\sdot,10) & (n \geq 1, m \geq 0) \,, 
\label{fnm-g}
\\[+3pt]
f(n,m,10,\sdot) &= (+\nu) f(n-1,m,\sdot,\sdot) \nonumber\\ &- m f(n-1,m,\sdot,01) & (n \geq 1, m \geq 0) \,,
\label{fnm-h}
\\[+3pt]
   f(n,m,11^A,\sdot)
&= \mu \nu f(n-2,m,\sdot,\sdot) \nonumber\\
&- \mu m f(n-2,m,\sdot,01) \nonumber\\
&- \nu m f(n-2,m,\sdot,10) \nonumber\\
&+ m f(n-2,m-1,\sdot,\sdot) \nonumber\\
&+ m(m-1) f(n-2,m,\sdot,11^A) & (n \geq 2, m \geq 0) \,,
\label{fnm-i}
\\[+3pt]
   f(n,m,11^C,\sdot)
&= m f(n-2,m-1,\sdot,\sdot) \nonumber\\
&+ m(m-1) f(n-2,m,\sdot,11^C) & (n \geq 2, m \geq 0) \,.
\label{fnm-j}
\end{align}
(When evaluating these recursive relations,
there may arise some undefined terms with negative arguments,
but this poses no problem as these terms are always accompanied 
by the factor zero. Similar remarks apply to the formulas below.)
Together with the initial conditions
\begin{align}
\label{fnm-o}
f(n,0) = (\mu \nu)^{n} \,, \qquad f(0,m) = (\mu \nu)^{m} \,,
\end{align}
the equations \eqref{fnm-a} -- \eqref{fnm-j} determine the values $f(n,m)$ completely.

\pagebreak[2]

We will now derive a recursive equation involving the values $f(n,m)$ only.
\linebreak By \eqref{fnm-i} and \eqref{fnm-j} with $m$ replaced by $m-1$, 
we have, for $n \ge 2$, $m \geq 1$,
\begin{align*}
& \mskip24mu f(n,m-1,11^A,\sdot) + f(n,m-1,11^C,\sdot) \\
&= \mu \nu f(n-2,m-1)
 - \mu (m-1) f(n-2,m-1,\sdot,01)
 - \nu (m-1) f(n-2,m-1,\sdot,10) \\
&+ 2(m-1) f(n-2,m-2)
 + (m-1)(m-2) \Big( f(n-2,m-1,\sdot,11^A) + f(n-2,m-1,\sdot,11^C) \Big) \,.
\end{align*}
By (\ref{fnm-f}) with $n,m$ replaced by $n-1,m-1$,
the two summands involving $11^A,11^C$ satisfy, 
for $n \geq 2$, $m \geq 1$,
\begin{align*}
& \mskip24mu (m-1)(m-2) \Big( f(n-2,m-1,\sdot,11^A) + f(n-2,m-1,\sdot,11^C) \Big) \\
&= - \mu \nu f(n-2,m-1)
 + \mu (m-1) f(n-2,m-1,\sdot,01)
 + \nu (m-1) f(n-2,m-1,\sdot,10) \\
&+ f(n-1,m-1)
 - (2b+\tfrac12) (m-1) f(n-2,m-2) \,,
\end{align*}
whence, for $n \geq 2$, $m \geq 1$,
\begin{align}
& \mskip24mu f(n,m-1,11^A,\sdot) + f(n,m-1,11^C,\sdot) \nonumber\\
&= 2(m-1) f(n-2,m-2)
 + f(n-1,m-1)
 - (2b+\tfrac12) (m-1) f(n-2,m-2) \,.
\label{fnm-tmpa}
\end{align}
Plugging this into \eqref{fnm-a} and rearranging terms, 
we find that, for $m \geq 1$,
\begin{multline}
   f(n,m)
 = \mu \nu f(n,m-1)
 - \mu n f(n,m-1,01,\sdot)
 - \nu n f(n,m-1,10,\sdot) \\
 + \big( n(n+1) + n (2b-\tfrac32) \big) f(n-1,m-1) 
 - \big( n(n-1)(m-1) (2b-\tfrac{3}{2}) \big) f(n-2,m-2) \,.
\label{fnm-tmpb}
\end{multline}
(For $n = 0$ and $n = 1$, \eqref{fnm-tmpb} follows directly from \eqref{fnm-a}.)
Thus, we~have eliminated the terms involving $11^A$ and $11^C$.

The terms involving $01$ and $10$ can be eliminated by a similar substitution.
\linebreak[2] Using \eqref{fnm-g} and \eqref{fnm-h} with $m$ replaced by $m-1$, 
\eqref{fnm-b} and \eqref{fnm-c} with $n,m$ replaced by $n-1,m-1$
and finally \eqref{fnm-tmpb} with $n,m$ replaced by $n-1,m-1$, 
we~have, for $n \geq 1$, $m \geq 1$,
\begin{align*}
&\mskip24mu - \Big( \mu f(n,m-1,01,\sdot) + \nu f(n,m-1,10,\sdot) \Big) \\
&= - (\mu^2 + \nu^2) f(n-1,m-1) + (m-1) \Big( \mu f(n-1,m-1,\sdot,10) + \nu f(n-1,m-1,\sdot,01) \Big) \\
&= - (\mu^2 + \nu^2) f(n-1,m-1) + 2 \mu \nu (m-1) f(n-1,m-2) \\ & \qquad \,-\, (n-1)(m-1) \Big( \mu f(n-1,m-2,01,\sdot) + \nu f(n-1,m-2,10,\sdot) \Big) \\
&= - (\mu^2 + \nu^2) f(n-1,m-1) + \mu \nu (m-1) f(n-1,m-2) \\
& \qquad \,+\, (m-1) f(n-1,m-1) \\ 
& \qquad \,-\, \big( n(n-1)(m-1) \,+\, (n-1)(m-1) (2b-\tfrac32) \big) f(n-2,m-2) \\
& \qquad \,+\, \big( (n-1)(n-2)(m-1)(m-2) (2b-\tfrac{3}{2}) \big) f(n-3,m-3) \,.
\end{align*}
Plugging this into \eqref{fnm-tmpb} and rearranging terms, 
it follows that, for $m \geq 1$,
\begin{align}
   f&(n,m)
 = n(n+m) f(n-1,m-1)
 - n^2(n-1)(m-1) f(n-2,m-2) \nonumber\\ 
&+ (2b-\tfrac32) n \Big( f(n-1,m-1) - 2(n-1)(m-1) f(n-2,m-2) \nonumber\\&\mskip224mu + (n-1)(n-2)(m-1)(m-2) f(n-3,m-3) \Big) \nonumber\\
&+ \mu \nu f(n,m-1) 
 + \mu \nu n (m-1) f(n-1,m-2)
 - (\mu^2 + \nu^2) n f(n-1,m-1) \,.
\label{fnm-1}
\end{align}
(For $n = 0$, \eqref{fnm-1} follows immediately from \eqref{fnm-tmpb}.)
By symmetry, we~also~have, for $n \geq 1$,
\begin{align}
   f&(n,m)
 = (n+m)m f(n-1,m-1)
 - (n-1)m^2(m-1) f(n-2,m-2) \nonumber\\ 
&+ (2b-\tfrac32) m \Big( f(n-1,m-1) - 2(n-1)(m-1) f(n-2,m-2) \nonumber\\&\mskip224mu + (n-1)(n-2)(m-1)(m-2) f(n-3,m-3) \Big) \nonumber\\
&+ \mu \nu f(n-1,m) 
 + \mu \nu (n-1) m f(n-2,m-1)
 - (\mu^2 + \nu^2) m f(n-1,m-1) \,.
\label{fnm-2}
\end{align}
Hence, the values $f(n,m)$ may be computed recursively starting from \eqref{fnm-o}
and using either \eqref{fnm-1} or \eqref{fnm-2}.

From now on, for $n \geq 0, m \geq 0$, let $f(n,m)$ denote 
the second-order correlation function of the characteristic polynomial 
of the (complex) sample covariance matrix as defined in the Introduction.
(For the rest of this section, we will usually omit the parameters $\mu,\nu$,
which are regarded as fixed.)
Then, by Lemma \ref{chiral}, we have to make the replacements
$f(n,m) \mapsto f(n,m) \cdot (\mu \nu)^{n-m}$, $\mu^2 \mapsto \mu$, $\nu^2 \mapsto \nu$ 
in the~pre\-ceding two equations, thereby obtaining
\begin{align}
   f&(n,m)
 = n(n+m) f(n-1,m-1)
 - n^2(n-1)(m-1) f(n-2,m-2) \nonumber\\
&+ (2b-\tfrac32) n \Big( f(n-1,m-1) - 2(n-1)(m-1) f(n-2,m-2) \nonumber\\&\mskip224mu + (n-1)(n-2)(m-1)(m-2) f(n-3,m-3) \Big) \nonumber\\
&+ \mu \nu f(n,m-1) 
 + \mu \nu n (m-1) f(n-1,m-2) 
 - (\mu + \nu) n f(n-1,m-1) \qquad
\label{fnm-3}
\end{align}
for $m \geq 1$ and
\begin{align}
   f&(n,m)
 = (n+m)m f(n-1,m-1)
 - (n-1)m^2(m-1) f(n-2,m-2) \nonumber\\
&+ (2b-\tfrac32) m \Big( f(n-1,m-1) - 2(n-1)(m-1) f(n-2,m-2) \nonumber\\&\mskip224mu + (n-1)(n-2)(m-1)(m-2) f(n-3,m-3) \Big) \nonumber\\
&+ f(n-1,m) 
 + (n-1) m f(n-2,m-1) 
 - (\mu + \nu) m f(n-1,m-1) \qquad
\label{fnm-4}
\end{align}
for $n \geq 1$, respectively, 
along with the initial conditions
\begin{align}
f(n,0) = 1 \qquad \text{and} \qquad f(0,m) = (\mu \nu)^m \,.
\label{fnm-9}
\end{align}
(Observe that in contrast to the equations for the chiral matrix, the equations 
for the sample covariance matrix are not fully symmetrical in $n$ and $m$.)

Let $\alpha := n - m$ denote the difference of $n$ and $m$
and assume that $\alpha \geq 0$.
(The~case $n \leq m$ could be reduced to this case by exchanging the roles 
of $n$~and~$m$ and by multiplying with the appropriate power of $\mu\nu$.)
We~will determine, for any $\alpha \in \mathbb{N}$,
the generating function
\begin{align}
\label{falphadefinition}
\sum_{m=0}^{\infty} \frac{f(m+\alpha,m)}{(m+\alpha)! \, m!} \, z^m \,,
\end{align}
where $|z| < 1$.

Put
$$
E(z) := \exp \left( -(\mu+\nu) \frac{z}{1-z} + (2b-\tfrac32)z \right) \,.
$$
We will show that for any $\alpha \in \mathbb{N}$,
the generating function $\eqref{falphadefinition}$ is given by
\begin{align}
\label{smartguess}
F_\alpha(z) := E(z) \, (1-z)^{-2-\alpha} \, \sum_{k=0}^{\infty} \frac{(\mu \nu z)^k}{(k+\alpha)! \, k!} (1-z)^{-2k} \,.
\end{align}
%
%
To begin with, observe that $F_\alpha(z)$ defines an analytic function 
on the unit disc in the complex plane, with power series representation
\begin{align}
\label{powerseries}
F_\alpha(z) = \sum_{m=0}^{\infty} c_\alpha(m) \, z^m \,,
\end{align}
say. Next, differentiating (\ref{smartguess}) \wrt~$z$, we~obtain
\begin{multline*}
F_\alpha'(z) = \left( - \frac{\mu+\nu}{(1-z)^2} + (2b-\tfrac32) \right) F_\alpha(z)
+ \frac{2+\alpha}{1-z} F_\alpha(z) \\
+ E(z) (1-z)^{-2-\alpha} \sum_{k=1}^{\infty} \frac{(\mu \nu z)^k}{(k+\alpha)! \, (k-1)!} (1-z)^{-2k} \left( \frac{1}{z} + \frac{2}{1-z} \right) \,.
\end{multline*}
Since
$$
\frac{1}{z} + \frac{2}{1-z} = \frac{1-z+2z}{z(1-z)} = \frac{1+z}{z(1-z)} \,,
$$
it follows that
\begin{align}
F_\alpha'(z) &= \frac{2+\alpha}{1-z} F_\alpha(z) + (2b-\tfrac32) F_\alpha(z) - \frac{\mu+\nu}{(1-z)^2} F_\alpha(z) + \mu \nu \frac{1+z}{(1-z)^2} F_{\alpha+1}(z) \,,
\label{ode-1}
\end{align}
or (equivalently)
\begin{align}
F_\alpha'(z) &= (2z F_\alpha'(z) + (2+\alpha) F_\alpha(z)) - (z^2 F_\alpha'(z) + (2+\alpha) z F_\alpha(z)) 
\nonumber\\ & \quad\quad\quad \,+\, (2b-\tfrac32) (1-z)^2 F_\alpha(z) - (\mu+\nu) F_\alpha(z) + \mu \nu (1+z) F_{\alpha+1}(z) \,.
\label{ode-2}
\end{align}
In terms of the power series coefficients $c_\alpha(m)$ 
defined by \eqref{powerseries}, this translates into
the recursive relation, for $m \geq 1$,
\begin{align}
   m c_\alpha(m)
&= (2m+\alpha) c_\alpha(m-1)
 - (m+\alpha) c_\alpha(m-2) \nonumber \\
&+ (2b-\tfrac32) 
   \Big( c_\alpha(m-1) - 2c_\alpha(m-2) + c_\alpha(m-3) \Big) \nonumber \\
&+ \mu \nu c_{\alpha+1}(m-1) 
 + \mu \nu c_{\alpha+1}(m-2) 
 - (\mu + \nu) c_\alpha(m-1) \,,
\label{c-recursion-1}
\end{align}
where terms with a negative argument are to be regarded as zero.
For $m = 0$, we~clearly~have
\begin{align}
c_\alpha(0) = 1/\alpha! \,.
\label{c-recursion-0}
\end{align}
Comparing \eqref{fnm-3}, \eqref{fnm-9} and \eqref{c-recursion-1}, \eqref{c-recursion-0},
it is easy to see that the coefficients $c_\alpha(m)$ satisfy the same recursive equations
as the values
$
f(m+\alpha,m) / ((m+\alpha)! \, m!) \,.
$
Thus, as all values are uniquely determined by these recursive equations, it follows
that the generating function \eqref{falphadefinition} is given by \eqref{smartguess},
for any $\alpha \in \mathbb{N}$. 

\pagebreak[2]

Let us summarize our result for the complex setting as follows:

\begin{proposition}
\label{complex-SCM}
For any $\alpha \in \mathbb{N}$,
the second-order correlation function $f(n,m)$ of the characteristic polynomial 
of an unrescaled complex sample covariance matrix
satisfies
\begin{multline}
\sum_{m=0}^{\infty} \frac{f(m+\alpha,m)}{(m+\alpha)! \, m!} z^m = F_\alpha(z) \\
:= \exp \left( -(\mu+\nu) \frac{z}{1-z} + b^*z \right) \sum_{k=0}^{\infty} \frac{(\mu \nu z)^k}{(k+\alpha)! \, k!} (1-z)^{-2k-\alpha-2} \,,
\label{complex-SCMF}
\end{multline}
where $b^* := 2(b - \tfrac{3}{4})$.
\end{proposition}

\pagebreak[2]

Also, let us state the analogous result for the real setting:

\begin{proposition}
\label{real-SCM}
For any $\alpha \in \mathbb{N}$,
The second-order correlation function $f(n,m)$ of the characteristic polynomial 
of an unrescaled real sample covariance matrix
satisfies
\begin{multline}
\sum_{m=0}^{\infty} \frac{f(m+\alpha,m)}{(m+\alpha)! \, m!} z^m = F_\alpha(z) \\
:= \exp \left( -(\mu+\nu) \frac{z}{1-z} + b^*z \right) \sum_{k=0}^{\infty} \frac{(\mu \nu z)^k}{(k+\alpha)! \, k!} (1-z)^{-2k-\alpha-3} \,,
\label{real-SCMF}
\end{multline}
where $b^* := b - 3$.
\end{proposition}

As already mentioned at the beginning of this section,
the derivation in the real setting is essentially the same 
as that in the complex setting. 
That is why we do not give the full details of the proof 
of Proposition \ref{real-SCM}, but only mention 
some note\-worthy changes:

(i) 
In the real case, we have
$$
\ee X_{ij} = 0 \,, \qquad
\ee X_{ij}^2 = 1 \,, \qquad
\ee X_{ij}^4 = b \,.
$$
In particular, as the second moment is not equal to zero anymore, we get 
an extra term $n(n-1)f(n,m-1,11_B,\,\cdot\,)$ in Equation \eqref{fnm-x},
and
an extra term $m(m-1) \linebreak[2] f(n-1,m,\,\cdot\,,11_B)$ in Equation \eqref{fnm-y}.
Also, the factor $(2b+\tfrac12)$ must be replaced with the factor $b$
everywhere.

(ii)
Similar changes arise in Equations \eqref{fnm-a} -- \eqref{fnm-j}.
The recursive equation for the $11_B$-terms is the same as that for the $11_C$-terms, 
except that all occurrences of $11_C$ must be replaced with occurrences of $11_B$.
In fact, it is not hard to see that $f(n,m,11_B,\,\cdot\,) = f(n,m,11_C,\,\cdot\,)$
and $f(n,m,\,\cdot\,,11_B) = f(n,m,\,\cdot\,,11_C)$ for all $n,m \geq 0$.

(iii)
The analogue of \eqref{fnm-tmpa} reads
\begin{align*}
& \mskip24mu f(n,m-1,11^A,\sdot) + f(n,m-1,11^B,\sdot) + f(n,m-1,11^C,\sdot) \\
&= 3(m-1) f(n-2,m-2)
 + f(n-1,m-1)
 - b (m-1) f(n-2,m-2) \,.
\end{align*}
Consequently, the analogue of \eqref{fnm-tmpb} reads
\begin{multline*}
   f(n,m)
 = \mu \nu f(n,m-1)
 - \mu n f(n,m-1,01,\sdot)
 - \nu n f(n,m-1,10,\sdot) \\
 + \big( n(n+2) + n (b-3) \big) f(n-1,m-1) 
 - \big( n(n-1)(m-1) (b-3) \big) f(n-2,m-2) \,,
\end{multline*}
and the analogue of \eqref{fnm-1} reads
\begin{align*}
   f&(n,m)
 = n(n+m+1) f(n-1,m-1)
 - (n+1)n(n-1)(m-1) f(n-2,m-2) \\ 
&+ (b-3) n \Big( f(n-1,m-1) - 2(n-1)(m-1) f(n-2,m-2) \nonumber\\&\mskip224mu + (n-1)(n-2)(m-1)(m-2) f(n-3,m-3) \Big) \\
&+ \mu \nu f(n,m-1) 
 + \mu \nu n (m-1) f(n-1,m-2)
 - (\mu^2 + \nu^2) n f(n-1,m-1) \,.
\end{align*}
Similar modifications are required for equations \eqref{fnm-1} -- \eqref{fnm-4}.

(iv)
Defining $F_\alpha(z)$ as in Proposition \eqref{real-SCM}
and denoting by $c_\alpha(m)$ the coefficients in the power series representation
$F_\alpha(z) = \sum_{m=0}^{\infty} c_\alpha(m) \, z^m$,
it follows by the~same arguments as for \eqref{c-recursion-1} 
that for $m \geq 1$,
\begin{align*}
   m c_\alpha(m)
&= (2m+\alpha+1) c_\alpha(m-1)
 - (m+\alpha+1) c_\alpha(m-2) \nonumber \\
&+ (b-3) 
   \Big( c_\alpha(m-1) - 2c_\alpha(m-2) + c_\alpha(m-3) \Big) \nonumber \\
&+ \mu \nu c_{\alpha+1}(m-1) 
 + \mu \nu c_{\alpha+1}(m-2) 
 - (\mu + \nu) c_\alpha(m-1) \,.
\end{align*}
Similarly as in the complex setting, this is the same recursive relation
as that for the values $f(m+\alpha,m) / ((m+\alpha)! \, m!) \,.$

%
%

\bigskip

\section{Outline of the Asymptotic Analysis}

From now on, we will always assume that 
$n \equiv n_N := N$ and $m \equiv m_N := N - \alpha$
depend on $N$, where $\alpha$ is a fixed natural number.
In particular, $m_N / n_N \leq 1$ and $m_N / n_N \to 1$ as $N \to \infty$.

Let us begin with the complex setting.
It is well-known from random matrix theory that under the above assumptions, 
the spectrum of the properly rescaled sample covariance matrix $\tfrac{1}{n} X^* X$
is asymptotically concentrated on the interval $[0,4]$, 
with distribution given by the \emph{Mar\v{c}enko-Pastur density}
\begin{align}
\label{mp}
g(\xi) = \frac{1}{2\pi\xi} \sqrt{\xi(4-\xi)} \qquad (\xi \in (0,4)) \,.
\end{align}
Moreover, it is well-known from random matrix theory 
that the following 3 regions deserve particular attention:
\begin{enumerate}[(i)]
\item \textbf{the bulk of the spectrum},\\
\ie the region around a point $\xi \in (0,4)$,
\item \textbf{the soft edge of the spectrum},\\
\ie the region around the point $\xi = 4$,
\item \textbf{the hard edge of the spectrum},\\
\ie the region around the point $\xi = 0$.
\end{enumerate}
More precisely, it is widely expected that in these regions,
the correlation function of the eigenvalues
(see \textsc{Mehta} \cite{Me} or \textsc{Forrester} \cite{Fo})
is asymptotically given (after the appropriate rescaling
so that the mean spacing between the eigenvalues is of order $1$) by 
\begin{enumerate}[(i)]
\item
the sine kernel
\begin{align}
\label{sine}
\mathbb{S}(x,y) := \frac{\sin \pi(x-y)}{\pi(x-y)} \,,
\end{align}
\item
the Airy kernel
\begin{align}
\label{airy}
\mathbb{A}(x,y) := \frac{\ai(x) \ai'(y) - \ai'(x) \ai(y)}{x-y} \,,
\end{align}
\item
the Bessel kernel
\begin{align}
\label{bessel}
\mathbb{J}_{\alpha}(x,y) := \frac{J_\alpha(\sqrt{x}) \sqrt{y} J_\alpha'(\sqrt{y}) - \sqrt{x} J_\alpha'(\sqrt{x}) J_\alpha(\sqrt{y})}{2(x-y)} \,.
\end{align}
\end{enumerate}
Typically, these results were first obtained for Wishart matrices
(\ie,~Gaussian sample covariance matrices)
and then extended to more general sample covariance matrices:
\begin{enumerate}[(i)]
\item 
For the bulk of the spectrum, \textsc{Ben Arous} and \textsc{P\'ech\'e} \cite{BP} 
established the emergence of the sine kernel for sample covariance matrices $X^* X$
such that the distributions of the entries of the matrix $X$ are Gaussian convolutions.
\item 
For the soft edge of the spectrum, \textsc{Soshnikov} \cite{So} 
established the emergence of the Airy kernel by means of the method of moments
and a sophisticated comparison with the Gaussian case.
\item
For the hard edge of the spectrum, \textsc{Tao} and \textsc{Vu} \cite{TV} 
established the emergence of the Bessel kernel (in the special case where
the matrix $X$ is a square~matrix), also via a sophisticated reduction 
to the Gaussian case.
\end{enumerate}

It is the purpose of the following sections to show that the same kernels
show up at the level of the characteristic polynomial,
for quite general sample covariance matrices.
Thus, our results add some support to the universality conjecture 
that the kernels \emph{always} occur,
irrespective of the choice of the underlying distribution~$Q$.
Also, our proofs are comparatively simple,
and they are based on rather weak moment conditions,
considerably weaker ones than those for the above-mentioned results 
from the literature.

The starting point for our asymptotic analysis 
will also be the integral representation
\begin{multline}
\label{intrep1}
\ \ \frac{f(n,m;\mu,\nu)}{n! \, m!}
=
\frac{1}{2 \pi i} \int_\gamma
\exp \left( -(\mu+\nu) \frac{z}{1-z} + b^* z \right) 
\\
\,\cdot\, \sum_{k=0}^{\infty} \frac{(\mu \nu z)^k}{(k+\alpha)! \, k!} (1-z)^{-2k-\alpha-2}
 \ \frac{dz}{z^{m+1}} \,, \ \ 
\end{multline}
which follows from \eqref{complex-SCMF} by Cauchy's formula.
Here, $b^*$ is defined as in \eqref{complex-SCMF}, $\alpha := n - m$, 
and $\gamma$ denotes a~counterclockwise circular path 
of radius $R < 1$, $R \approx 1$ around the origin.
(The precise choice of the radius will be specified 
in the~following sections.)
Since
\begin{align*}
\sum_{k=0}^{\infty} \frac{(\mu \nu z)^k}{(k+\alpha)! \, k!} (1-z)^{-2k}
=
\left( \frac{(\mu \nu z)^{1/2}}{1-z} \right)^{-\alpha} \, I_\alpha \left( \frac{2 (\mu \nu z)^{1/2}}{1-z} \right) \,,
\end{align*}
where $I_\alpha$ denotes the modified Bessel function of order $\alpha$,
\eqref{intrep1} may be rewritten as
\begin{align}
\label{intrep2}
\frac{f(n,m;\mu,\nu)}{n! \, m!}
=
\frac{1}{2 \pi i} \int_\gamma
\frac{\exp \left( -(\mu+\nu) \frac{z}{1-z} + b^* z \right) \cdot I_\alpha \left( \frac{2 (\mu \nu z)^{1/2}}{1-z} \right)}{(1-z)^2 \cdot ((\mu \nu z)^{1/2})^\alpha}
 \ \frac{dz}{z^{m+1}} \,.
\end{align}
We will always use the parametrization $\gamma(t) := Re^{it}$,
$-\pi \leq t \leq +\pi$, and we will always assume that $\sqrt{z}$ 
is chosen such that $\arg \sqrt{z} = 0$ on the positive real axis
and $\arg \sqrt{z}$ is continuous otherwise,
so $\arg \sqrt{z} \in [-\pi/2,+\pi/2]$ as $z$ traverses
the~path $\gamma(t) := Re^{it}$, $-\pi \leq t \leq +\pi$.
Similar remarks apply to the choice of the~branch 
of the modified Bessel function.

In the course of our asymptotic analysis, we will use 
the following classical results about the asymptotic behavior
of the modified Bessel function: For fixed $\alpha \geq 0$,
\begin{align}
\label{bessel-a1}
I_\alpha(z) = \frac{\exp(z)}{\sqrt{2\pi z}} \left( 1 + \myo_{\varepsilon}(1/z) \right) \qquad (|z| \to \infty, |\arg z| \leq \tfrac{\pi}{2} - \varepsilon)
\end{align}
and
\begin{align}
\label{bessel-a2}
|I_\alpha(z)| \leq C \, \frac{\exp(\re z)}{\sqrt{|z|}} \qquad (|z| \to \infty, |\arg z| \leq \tfrac{\pi}{2})
\end{align}
where $C > 0$ is an absolute constant.
(See \eg Chapter~7 in \textsc{Olver} \cite{O}.)
In fact, our results could even be extended to the more general situation 
where $m_N / n_N$ tends to some constant $\gamma \in (0,1)$ as $N \to \infty$.
However, this requires a more careful treatment of remainder terms
(particularly the remainder terms originating from the modified Bessel functions)
and will therefore not be pursued here.

We will show that in all the above-mentioned cases, 
given the appropriate specialization 
of the shift parameters $\mu,\nu$ and the radius $R$,
the main contribution to the integral
comes from a small neighborhood of the point $z = R$
and that we~asymptotically end up with the above-mentioned kernels
from random matrix theory.

Furthermore, by essentially the same proofs, 
we obtain similar results for real sample covariance matrices.
In this setting the correlation function of the characteristic polynomial
is asymptotically given (after the appropriate rescaling) by 
\begin{enumerate}[(i)]
\item
the ``differentiated'' sine kernel
\begin{align}
\label{sine2}
\widetilde{\mathbb{S}}(x,y) := \frac{2\sin \pi(x-y)}{\pi(x-y)^3} - \frac{2\cos \pi(x-y)}{(x-y)^2} \,,
\end{align}
\item
the ``differentiated'' Airy kernel
\begin{multline}
\label{airy2}
\qquad \widetilde{\mathbb{A}}(x,y) := \frac{2 \ai(x) \ai'(y) - 2 \ai'(x) \ai(y)}{(x-y)^3} \\ + \frac{(x+y) \ai(x) \ai(y) - 2 \ai'(x) \ai'(y)}{(x-y)^2} \,, \quad
\end{multline}
\item
the ``differentiated'' Bessel kernel
\begin{multline}
\label{bessel2}
\qquad \widetilde{\mathbb{J}}_{\alpha}(x,y) := \frac{(x+y) \big( J_\alpha(\sqrt{x}) \sqrt{y} J_\alpha'(\sqrt{y}) - \sqrt{x} J_\alpha'(\sqrt{x}) J_\alpha(\sqrt{y}) \big)}{2(x-y)^3} \\ - \frac{(x+y-2\alpha^2) \, J_\alpha(\sqrt{x}) J_\alpha(\sqrt{y}) + 2\sqrt{xy} \, J_\alpha'(\sqrt{x}) J_\alpha'(\sqrt{y})}{4(x-y)^2} \,. \quad
\end{multline} 
\end{enumerate}
These kernels 
are obtained from the corresponding kernels for the complex setting
by applying the differential operators 
$D := \frac{1}{x-y} \left( \frac{\partial}{\partial y} - \frac{\partial}{\partial x} \right)$
(in~cases (i) and (ii)) and
$D := \frac{1}{x-y} \left( y \, \frac{\partial}{\partial y} - x \, \frac{\partial}{\partial x} \right)$
(in~case (iii)).
(Besides that, to~obtain the above formulas for the ``differentiated'' kernels,
we have to use the~differential equation for the Airy~and~Bessel function.)
Since the derivation of the results for real sample covariance matrices
is essentially the same as for complex sample covariance matrices, 
we will refrain from giving the details of the~proofs, 
but only state the final results.

\bigskip

\section{Asymptotics in the Bulk of the Spectrum}

In order to zoom in around a point in the bulk of the spectrum,
we have to make the replacements
$$
\mu \mapsto N \xi + \mu/g(\xi)
\quad\text{and}\quad
\nu \mapsto N \xi + \nu/g(\xi)
$$
for some $\xi \in (0,4)$, $\mu,\nu \in \mathbb{R}$.
(Recall that our sample covariance matrices are unrescaled and that 
$g(\xi)$ denotes density of the Mar\v{c}enko-Pastur distribution.)
Moreover, we have to multiply by the proper scaling factor in order 
to obtain a~non-degenerate limit. We then have the following result:

\begin{theorem}
\label{complex-S}
Let $f$ denote the second-order correlation function of a complex sample covariance matrix
satisfying our standing moment conditions.
For any $\alpha \in \mathbb{N}$, $\xi \in (0,4)$, $\mu,\nu \in \mathbb{R}$, setting
$$
Z_N(\xi,\mu,\nu) := \left( N^2 \xi^2 + N \xi (\mu + \nu) + \mu \nu \right)^{\alpha/2} \, \exp \left( -N\xi - \tfrac12(\mu+\nu) \right) \,,
$$
we have
\begin{multline}
\label{complex-SF}
\lim_{N \to \infty} \left( g(\xi)^{-1} \, Z_N(\xi,\mu/g(\xi),\nu/g(\xi)) \cdot \frac{f(n,m;N\xi+\mu/g(\xi),N\xi+\nu/g(\xi))}{n! \, m!} \right) \\
= \exp(b^*) \, \mathbb{S}(\mu,\nu) \,,
\end{multline}
where $b^* := 2(b-\tfrac34)$ and $\mathbb{S}$ is defined as in \eqref{sine}.
\end{theorem}

\begin{proof}[Proof of Theorem \ref{complex-S}]
By (\ref{intrep2}), we have
\begin{multline}
\label{intrep11}
Z_N(\xi,\mu,\nu) \cdot \frac{f(n,m;N\xi+\mu,N\xi+\nu)}{n! \, m!} \\ = 
\frac{1}{2 \pi i} \int_\gamma
\frac{\exp \left( -(N\xi+\tfrac12(\mu+\nu)) \frac{1+z}{1-z} + b^*z \right) \cdot I_\alpha \left( w(z) \right)}{(1-z)^2 \cdot z^{(n+m)/2}}
\ \frac{dz}{z} \,, 
\end{multline}
where we have put
$$
w(z) = w_N(z) := 2 \sqrt{N^2 \xi^2 + N \xi (\mu + \nu) + \mu \nu} \cdot \frac{\sqrt{z}}{1-z}
$$
for abbreviation. 
We take $R := 1 - 1/N$.
We reduce the proof of Theorem \ref{complex-S} to~the following two claims:

On the one hand, we will show that, for any (fixed) $a > 0$,
\begin{align}
\label{claim11}
\frac{1}{2 \pi i} \int_{\gamma'}
\frac{\exp \left( -(N\xi+\tfrac12(\mu+\nu)) \frac{1+z}{1-z} + b^*z \right) \cdot I_\alpha(w(z))}{(1-z)^2 \cdot z^{(n+m)/2}}
 \ \frac{dz}{z}
\overset{N \to \infty}{-\!\!\!-\!\!\!-\!\!\!\to} S(a) \,,
\end{align}
where $\gamma' \equiv \gamma'(a)$ denotes the restriction of the path $\gamma$
to the set $(-a/N,+a/N)$, and
$$
S(a) := \frac{\exp(b^*)}{4\pi^{3/2}\xi^{1/2}} \int_{-a}^{+a} \frac{\exp((1-\tfrac14\xi) \cdot (1-iu) - \tfrac14(\mu-\nu)^2\xi^{-1} / (1-iu))}{(1-iu)^{3/2}} \, du \,.
$$

On the other hand, we will show that for any $\delta > 0$, there exists 
some $a_0(\delta) > 0$ such that for $a \geq a_0(\delta)$, we have
\begin{align}
\label{claim12}
\left| \frac{1}{2 \pi i} \int_{\gamma''}
\frac{\exp \left( -(N\xi+\tfrac12(\mu+\nu)) \frac{1+z}{1-z} + b^*z \right) \cdot I_\alpha(w(z))}{(1-z)^2 \cdot z^{(n+m)/2}}
 \ \frac{dz}{z} \right|
\leq \delta
\end{align}
for all $N \in \mathbb{N}$ large enough.
Here, $\gamma'' \equiv \gamma''(a)$ denotes the restriction of the path $\gamma$
to the set $(-\pi,+\pi) \setminus (-a/N,+a/N)$.

It is easy to see that the main result \eqref{complex-SF} follows 
by combining \eqref{claim11} and \eqref{claim12}
and by using Laplace inversion.
Indeed, firstly, \eqref{claim11}~and~\eqref{claim12} imply that
$$
\frac{1}{2 \pi i} \int_{\gamma}
\frac{\exp \left( -(N\xi+\tfrac12(\mu+\nu)) \frac{1+z}{1-z} + b^*z \right) \cdot I_\alpha(w(z))}{(1-z)^2 \cdot z^{(n+m)/2}}
 \ \frac{dz}{z}
\overset{N \to \infty}{-\!\!\!-\!\!\!-\!\!\!\to} S(\infty) \,,
$$
where
$$
S(\infty) := \frac{\exp(b^*)}{4\pi^{3/2}i\xi^{1/2}} \int_{1-i\infty}^{1+i\infty} \frac{\exp((1-\tfrac14\xi) z - \tfrac14(\mu-\nu)^2\xi^{-1} / z)}{z^{3/2}} \, dz \,.
$$
Secondly, by Laplace inversion, for $t > 0$, $a \in \mathbb{R}$,
$$
\frac{1}{2\pi i} \int_{1-i\infty}^{1+i\infty} e^{tz} \, \frac{e^{-a^2/4z}}{z^{3/2}} \ dz 
= 
\frac{2 \sin (a \sqrt{t})}{\sqrt{\pi} a}
$$
(see \eg p.\,245 in \cite{BMP}), whence
$$
S(\infty) = \tfrac{1}{\pi} \exp(b^*) \left( (1-\tfrac{1}{4}\xi)/\xi \right)^{1/2} \cdot \frac{\sin \left( (\mu-\nu) \left( (1-\tfrac14\xi)/\xi \right)^{1/2} \right)}{\left( (\mu-\nu) \left( (1-\tfrac14\xi)/\xi \right)^{1/2} \right)} \,.
$$
Replacing the local shift parameters $\mu,\nu$ with $\mu/g(\xi),\nu/g(\xi)$,
dividing by $g(\xi)$ 
and noting that $\left( (1-\tfrac{1}{4}\xi)/\xi \right)^{1/2} = \pi g(\xi)$
yields \eqref{complex-SF}.

\pagebreak[2]
\medskip

To prove \eqref{claim11} and \eqref{claim12}, it turns out convenient
to rewrite the integrand in \eqref{intrep11} as 
$$
\frac{f_1(z) f_2(z)}{(1-z)^{2} \, z^{(n+m)/2}} \,,
$$
where
$$
f_1(z) := \exp \left( -(N\xi + \tfrac12(\mu+\nu)) \frac{1+z}{1-z} + b^*z \right) \cdot \exp \left( + w(z) \right)
$$
and
$$
f_2(z) := \exp \left( - w(z) \right) \cdot I_\alpha \left( w(z) \right) \,.
$$
Since
$$
w(z) = 2 \left( N \xi + \tfrac12 (\mu + \nu) - \tfrac18 (\mu - \nu)^2 / N \xi + \myo(1/N^2) \right) \frac{\sqrt{z}}{1-z}
$$
and
$$
- \frac{1+z}{1-z} + \frac{2\sqrt{z}}{1-z} = - \frac{(1-\sqrt{z})^2}{(1+\sqrt{z})(1-\sqrt{z})} = - \frac{1-\sqrt{z}}{1+\sqrt{z}} \,,
$$
we find that
\begin{multline*}
f_1(z) = \exp \bigg( -N\xi \frac{1-\sqrt{z}}{1+\sqrt{z}} - \tfrac12(\mu+\nu) \frac{1-\sqrt{z}}{1+\sqrt{z}} - \tfrac14 (\mu-\nu)^2 / (N\xi) \, \frac{\sqrt{z}}{1-z} \\ + b^*z + \myo (1/N^2) \, \frac{\sqrt{z}}{1-z} \bigg) \,.
\end{multline*}
Here and in the sequel, we take the convention 
that the implicit constants in the $\myo$-terms 
may depend on $\alpha,b^*,\xi,\mu,\nu$ (which are regarded as fixed).

\medskip

We first prove (\ref{claim11}). In doing so, we use the notation $\myo_a$ 
to denote a bound involving an implicit constant depending also on $a$ 
(in addition to $\alpha,b^*,\xi,\mu,\nu$).
Substituting $z = Re^{it}$ and $t = u/N$ on the left-hand side 
in \eqref{claim11}, we obtain
\begin{align}
\label{step11a}
\frac{1}{2 \pi N} \int_{-a}^{+a}
\frac{f_1(\gamma(u/N)) \, f_2(\gamma(u/N))}{(1-\gamma(u/N))^2 \cdot (\gamma(u/N))^{(n+m)/2}}
 \ du \,.
\end{align}
Now, by straightforward Taylor expansion, we have the approximations, 
for $|u| \leq a$,
\begin{align*}
f_1(Re^{iu/N}) &= \exp \big( -\tfrac14\xi(1-iu) - \tfrac14(\mu-\nu)^2 \xi^{-1} / (1-iu) + b^* \big) \left( 1 + \myo_a(N^{-1}) \right) \,,
\\
f_2(Re^{iu/N}) &= \frac{1}{\sqrt{4\pi N^2 \xi / (1-iu)}} \left( 1 + \myo_a(N^{-1}) \right) \,,
\\
(1-Re^{iu/N})^2 &= (1-iu)^2/N^2 \left( 1 + \myo_a(N^{-1}) \right) \,,
\\
(Re^{iu/N})^{(m+n)/2} &= \exp \big( - (1-iu) \big) \left( 1 + \myo_a(N^{-1}) \right) \,.
\end{align*}
For the second approximation, we have used the observation that for $|u| \leq a$,
$
w(Re^{iu/N}) = 2N^2 \xi / (1-iu) \left( 1 + \myo_a(N^{-1}) \right),
$
as well as the asymptotic approxi\-mation \eqref{bessel-a1} 
for the modified Bessel function.
Inserting the preceding approxi\-mations into \eqref{step11a}, 
it~follows that the left-hand side in \eqref{step11a} converges to
$$
\frac{\exp(b^*)}{4\pi^{3/2}\xi^{1/2}} \int_{-a}^{+a} \frac{\exp((1-\tfrac14\xi) \cdot (1-iu) - \tfrac14(\mu-\nu)^2\xi^{-1} / (1-iu))}{(1-iu)^{3/2}} \, du \,,
$$
which proves \eqref{claim11}.

\medskip

We now turn to the proof of (\ref{claim12}). 
It suffices to show that for $a \geq a_0(\delta)$,
\begin{align}
\label{step12a}
\frac{1}{2 \pi} \int_{[-\pi,-a/N] \,\cup\, [+a/N,+\pi]} \frac{|f_1(Re^{it})||f_2(Re^{it})|}{|1-Re^{it}|^2 \, |Re^{it}|^{(n+m)/2}} \ dt \leq \delta \,.
\end{align}
Write $\sqrt{Re^{it}} = re^{i\varphi}$ 
with $r = \sqrt{1 - 1/N}$, $\varphi \in [-\pi/2,+\pi/2]$.
Then,
$$
  \re \left( \frac{1-re^{i\varphi}}{1+re^{i\varphi}} \right)
= \frac{1 - r^2}{1 + r^2 + 2r \cos\varphi} \geq \frac{1}{4N} \,,
$$
$$
  \left| \frac{1-re^{i\varphi}}{1+re^{i\varphi}} \right|
= \frac{   \left| 1-re^{i\varphi} \right| }{ \left| 1+re^{i\varphi} \right| } 
\leq 1 \,,
$$
$$
\left| \frac{re^{i\varphi}}{1-(re^{i\varphi})^2} \right| \leq \frac{1}{1-r^2} = N  \,,
$$
from which it follows that 
\begin{align}
\label{step12b}
f_1(Re^{it}) = \myo(1) \,.
\end{align}
Also, for $N \in \mathbb{N}$ sufficiently large,
$$
\big| w(Re^{it}) \big| \geq N\xi \left| \frac{re^{i\varphi}}{1-(re^{i\varphi})^2} \right| \ge \tfrac12 N\xi \, \frac{1}{|1-Re^{it}|} \ge \tfrac{1}{4} N\xi 
$$
and
\begin{align*}
   \re \left( w(Re^{it}) \right) 
&= 2 \sqrt{N^2 \xi^2 + N \xi (\mu + \nu) + \mu \nu} \cdot \re \left( \frac{re^{i\varphi}}{1-(re^{i\varphi})^2} \right) \\
&= 2 \sqrt{N^2 \xi^2 + N \xi (\mu + \nu) + \mu \nu} \cdot \re \left( \frac{(r-r^3)\cos \varphi}{|1-Re^{it}|^2} \right)
\geq 0 \,,
\end{align*}
from which it follows using \eqref{bessel-a2} that 
\begin{align}
\label{step12c}
|f_2(Re^{it})| = \myo \left( |1-Re^{it}|^{1/2} /\, \sqrt{N} \right) \,.
\end{align}
Finally,
\begin{align}
\label{step12d}
|(Re^{it})^{-(n+m)/2}| = (1-1/N)^{-(N-\alpha/2)} = \myo(1) \,.
\end{align}
Thus, by \eqref{step12b}, \eqref{step12c}, \eqref{step12d} and symmetry, 
it remains to show that for $a \geq a_0(\delta)$,
\begin{align}
\label{step12e}
\frac{1}{2 \pi} \int_{a/N}^{\pi} \frac{1}{|1-Re^{it}|^{3/2}} \ dt \leq \delta N^{1/2} \,.
\end{align}
To this end, let $c > 0$ be a~small constant such that
$\cos t \leq 1 - c^2 t^2$ for all $|t| \leq \pi$.
Then we have
$$
  |1-Re^{it}|^2 
= (1-R\cos t)^2 + (R\sin t)^2 
= (1-R)^2 + 2R(1-\cos t)
\geq 
  c^2 t^2
$$
and therefore
\begin{align}
\label{stdest}
\int_{a/N}^{\pi} |1-Re^{it}|^{-3/2} \ dt
\leq
\int_{a/N}^{\infty} (ct)^{-3/2} \ dt
=
2(ca/N)^{-1/2} \,,
\end{align}
which entails \eqref{step12e} by choosing $a_0(\delta) > 0$ large~enough.

\medskip

This completes the proof of Theorem \ref{complex-S}.
\end{proof}

\pagebreak[2]

With essentially the same proof, we obtain the following result:

\begin{theorem}
\label{real-S}
Let $f$ denote the second-order correlation function of a real sample covariance matrix
satisfying our standing moment conditions.
For any $\alpha \in \mathbb{N}$, $\xi \in (0,4)$, $\mu,\nu \in \mathbb{R}$, setting
$$
Z_N(\xi,\mu,\nu) := \left( N^2 \xi^2 + N \xi (\mu + \nu) + \mu \nu \right)^{\alpha/2} \, \exp \left( -N\xi - \tfrac12(\mu+\nu) \right) \,,
$$
we have
\begin{multline}
\label{real-SF}
\lim_{N \to \infty} \left( N^{-1} \, \xi^{-1} \, \big( g(\xi) \big)^{-3} \, Z_N(\xi,\mu/g(\xi),\nu/g(\xi)) \cdot \frac{f(n,m;N\xi+\mu/g(\xi),N\xi+\nu/g(\xi))}{n! \, m!} \right) \\
= \exp(b^*) \, \widetilde{\mathbb{S}}(\mu,\nu) \,,
\end{multline}
where $b^* := (b-3)$ and $\widetilde{\mathbb{S}}$ is defined as in \eqref{sine2}.
\end{theorem}

\bigskip

\section{Asymptotics at the Soft Edge of the Spectrum}

In order to zoom in at the soft edge of the spectrum,
we have to make the replacements
$$
\mu \mapsto 4N + 2^{4/3} \mu N^{1/3} \quad \text{and} \quad \nu \mapsto 4N + 2^{4/3} \nu N^{1/3}
$$
where $\mu,\nu \in \mathbb{R}$. (The factor $2^{4/3}$ is for convenience.)
We then have the following result:

\begin{theorem}
\label{complex-A}
Let $f$ denote the second-order correlation function of a complex sample covariance matrix
satisfying our standing moment conditions.
For any $\alpha \in \mathbb{N}$, $\mu,\nu \in \mathbb{R}$, setting
$$
Z_N(\mu,\nu) := \left( 16N^2 + 4 (\mu + \nu) N^{4/3} + \mu \nu N^{2/3} \right)^{\alpha/2} \, \exp \left( - 4N - \tfrac12(\mu+\nu) N^{1/3} \right) \,,
$$
we have
\begin{multline}
\label{complex-AF}
\lim_{N \to \infty} \left( 2^{4/3} \, N^{1/3} \, Z_N(2^{4/3} \mu,2^{4/3} \nu) \cdot \frac{f(n,m;4N+2^{4/3} \mu N^{1/3},4N+2^{4/3} \nu N^{1/3})}{n! \, m!} \right) \\
= \exp(b^*) \, \mathbb{A}(\mu,\nu) \,, 
\end{multline}
where $b^* := 2(b-\tfrac34)$ and $\mathbb{A}$ is defined as in \eqref{airy}.
\end{theorem}

\begin{proof}[Proof of Theorem \ref{complex-A}]
By (\ref{intrep2}), we have
\begin{multline}
\label{intrep21}
N^{1/3} \, Z_N(\mu,\nu) \cdot \frac{f(n,m;4N+\mu N^{1/3},4N+\nu N^{1/3})}{n! \, m!} \\ = 
\frac{N^{1/3}}{2 \pi i} \int_\gamma
\frac{\exp \left( -(4N+\tfrac12(\mu+\nu)N^{1/3}) \frac{1+z}{1-z} + b^*z \right) \cdot I_\alpha \left( w(z) \right)}{(1-z)^2 \cdot z^{(n+m)/2}}
\ \frac{dz}{z} \,, 
\end{multline}
where now
$$
w(z) := 2 \sqrt{16N^2 + 4 (\mu+\nu) N^{4/3} + \mu \nu N^{2/3}} \cdot \frac{\sqrt{z}}{1-z} \,.
$$
This time we take $R := 1 - 1/N^{1/3}$.
Similarly as for the bulk of the spectrum, 
we~reduce the proof to the following two claims:

On the one hand, we will show that for any (fixed) $a > 0$,
\begin{align}
\label{claim21}
\frac{N^{1/3}}{2\pi i} \int_{\gamma'} \frac{\exp \left( - (4N + \tfrac12 (\mu+\nu) N^{1/3}) \frac{1+z}{1-z} + b^*z \right) \cdot I_\alpha(w(z))}{(1-z)^2 \cdot z^{(n+m)/2}} \, \frac{dz}{z}
\overset{N \to \infty}{-\!\!\!-\!\!\!-\!\!\!\to}
A(a) \,,
\end{align}
where $\gamma' \equiv \gamma'(a)$ denotes the restriction of the path $\gamma$
to the set $(-a/N^{1/3},+a/N^{1/3})$, and
$$
A(a) 
:=
\frac{\exp(b^*)}{8\pi^{3/2}} \int_{-a}^{+a} \frac{\exp \left( \tfrac{1}{48} (1-iu)^3 - \tfrac{1}{8} (\mu+\nu) (1-iu) - \tfrac{1}{16} (\mu-\nu)^2 (1-iu)^{-1} \right)}{(1-iu)^{3/2}} \, du \,.
$$

On the other hand, we will show that for any $\delta > 0$,
there exists some $a_0(\delta) > 0$  such that for $a \geq a_0(\delta)$,
we have
\begin{align}
\label{claim22}
\left|
\frac{N^{1/3}}{2\pi i} \int_{\gamma''} \frac{\exp \left( - (4N + \tfrac12 (\mu+\nu) N^{1/3}) \frac{1+z}{1-z} + b^*z \right) \cdot I_\alpha(w(z))}{(1-z)^2 \cdot z^{(n+m)/2}} \, \frac{dz}{z}
\right|
\leq 
\delta
\end{align}
for all $N \in \mynat$ large enough.
Here $\gamma'' \equiv \gamma''(a)$ denotes the restriction of the path $\gamma$
to the set $(-\pi,+\pi) \setminus (-a/N^{1/3},+a/N^{1/3})$.

Using the integral representation for the Airy kernel given in \cite{Ko2},
the theorem may then be deduced 
using similar arguments as in the previous section.
Indeed, first of all, \eqref{claim21} and \eqref{claim22} imply that
\begin{align*}
\frac{N^{1/3}}{2\pi i} \int_{\gamma} \frac{\exp \left( - (4N + \tfrac12 (\mu+\nu) N^{1/3}) \frac{1+z}{1-z} + b^*z \right) \cdot I_\alpha(w(z))}{(1-z)^2 \cdot z^{(n+m)/2}} \, \frac{dz}{z}
\overset{N \to \infty}{-\!\!\!-\!\!\!-\!\!\!\to}
A(\infty) \,,
\end{align*}
where
$$
A(\infty) =\frac{\exp(b^*)}{8\pi^{3/2}i} \int_{1-i\infty}^{1+i\infty} \frac{\exp \left( \tfrac{1}{48} z^3 - \tfrac{1}{8}(\mu+\nu)z - \tfrac{1}{16}(\mu-\nu)^2 z^{-1} \right)}{z^{3/2}} \, dz \,.
$$
Substituting $z = 2^{2/3} z$, $dz = 2^{2/3} dz$, 
and shifting the path back to the line $1 + i\mathbb{R}$
(which is easily justified by Cauchy's theorem), we obtain
$$
\frac{2^{2/3} \exp(b^*)}{16\pi^{3/2}i} \int_{1-i\infty}^{1+i\infty} \frac{\exp \left( \tfrac{1}{12}z^3 - \tfrac{1}{8}(\mu+\nu) \, 2^{2/3} z - \tfrac{1}{16}(\mu-\nu)^2 \, (2^{2/3} z)^{-1} \right)}{z^{3/2}} \, dz \,.
$$
Making the replacements $\mu \mapsto 2^{4/3} \mu$, $\nu \mapsto 2^{4/3} \nu$
and multiplying by $2^{4/3}$, \linebreak we further obtain
$$
\frac{\exp(b^*)}{4\pi^{3/2}i} \int_{1-i\infty}^{1+i\infty} \frac{\exp \left( \tfrac{1}{12}z^3 - \tfrac{1}{2}(\mu+\nu) - \tfrac{1}{4}(\mu-\nu)^2 z^{-1} \right)}{z^{3/2}} \, dz \,.
$$
By Proposition 2.2 in \cite{Ko2}, the latter expression
is equal to $\exp(b^*) \, \mathbb{A}(\mu,\nu)$,
whence \eqref{complex-AF}.

\medskip

To prove \eqref{claim21} and \eqref{claim22}, we proceed similarly
as in the last section. First of all, we rewrite the integrand as
$$
\frac{f_1(z) \, f_2(z)}{(1-z)^2 \, z^{(n+m)/2}} \,,
$$
where now
$$
f_1(z) := \exp \left( - \big( 4N + \tfrac12(\mu+\nu)N^{1/3} \big) \frac{1+z}{1-z} + b^*z \right) \cdot \exp \left( + w(z) \right)
$$
and
$$
f_2(z) := \exp \left( - w(z) \right) \cdot I_\alpha \left( w(z) \right) \,.
$$
Similarly as in the previous section, we find have
\begin{multline*}
f_1(z)
= 
\exp \bigg( - 4N \frac{1-\sqrt{z}}{1+\sqrt{z}} - \tfrac12(\mu+\nu) N^{1/3} \frac{1-\sqrt{z}}{1+\sqrt{z}} - \tfrac{1}{16} (\mu-\nu)^2 N^{-1/3} \frac{\sqrt{z}}{1-z} \\ + b^* z + \myo ( N^{-1} ) \, \frac{\sqrt{z}}{1-z} \bigg) \,.
\end{multline*}

\medskip

To prove (\ref{claim21}), we use Taylor expansion.
Straightforward calculations show that, for $|u| \leq a$,
\begin{align*}
f_1(\gamma(u/N^{1/3})) 
&= 
\exp \Big( - (1-iu) N^{2/3} - \tfrac{1}{2} N^{1/3} - \tfrac{1}{3} + \tfrac{1}{48} (1-iu)^3 \\&\qquad
 - \tfrac{1}{8} (\mu+\nu) (1-iu) - \tfrac{1}{16} (\mu-\nu)^2 / (1-iu) + b^* \Big) \left( 1 + \myo_a(N^{-1/3}) \right) \,,
\\
f_2(\gamma(u/N^{1/3}))
&= 
\frac{1}{\sqrt{16\pi N^{4/3}/(1-iu)}} \left( 1 + \myo_a(N^{-1/3}) \right) \,,
\\
(1-\gamma(u/N^{1/3}))^2
&=
(1-iu)^2 / N^{2/3} \left( 1 + \myo_a(N^{-1/3}) \right) \,,
\\
(\gamma(u/N^{1/3}))^{(m+n)/2}
&=
\exp \left( (1-iu)N^{2/3} + \tfrac12 N^{1/3} + \tfrac13 \right) \left( 1 + \myo_a(N^{-1/3}) \right) \,.
\end{align*}
For the second approximation, we have used the observation that for $|u| \leq a$,
$
w(Re^{iu/N}) = 8N^{4/3} / (1-iu) \left( 1 + \myo_a(N^{-1/3}) \right),
$
as well as the asymptotic \linebreak approximation \eqref{bessel-a1} 
for the modified Bessel function.
Putting it all together, the highest-order terms cancel out, 
and we find that the left-hand side of \eqref{claim21} is asymptotically given by
$$
\frac{\exp(b^*)}{8\pi^{3/2} i} \int_{a}^{a} \frac{\exp \left( \tfrac{1}{48} (1-iu)^3 - \tfrac{1}{8} (\mu+\nu) (1-iu) - \tfrac{1}{16} (\mu-\nu)^2 / (1-iu) \right)}{(1-iu)^{3/2}} \, du \,,
$$
which establishes \eqref{claim21}.

\medskip

It remains to show \eqref{claim22}. 
To this end, we will first show that
\begin{align}
\label{step22a}
\frac{f_1(Re^{it})}{(Re^{it})^{(m+n)/2}} = \myo(1) \,.
\end{align}
Write $\sqrt{Re^{it}} = re^{i\varphi}$,
where $r = \sqrt{1-1/N^{1/3}}$, $\varphi \in [-\pi/2,+\pi/2]$.
Then, since
$$
  \left| \re \left( \frac{1-re^{i\varphi}}{1+re^{i\varphi}} \right) \right|
= \left| \frac{1-r^2}{1+r^2+2r\cos\varphi} \right|  
\leq N^{-1/3} \,,
$$
$$
\left| \frac{re^{i\varphi}}{1-(re^{i\varphi})^2} \right| \leq \frac{1}{1-r^2} = N^{1/3}  \,,
$$
the proof of \eqref{step22a} is reduced to showing that
$$
\left| \frac{\exp \left( - 4N \frac{1-re^{i\varphi}}{1+re^{i\varphi}} \right)}{(Re^{it})^{(m+n)/2}} \right| = \myo(1) \,.
$$
Now, for sufficiently large $N \in \mathbb{N}$,
$$
  \re \left( \frac{1-re^{i\varphi}}{1+re^{i\varphi}} \right)
= \frac{1 - r^2}{1 + r^2 + 2r \cos\varphi} 
= \frac{N^{-1/3}}{2r(1+\cos\varphi) + (1-r)^2}\,, 
$$
$$
  \big| Re^{it} \big|^{(m+n)/2} 
= \exp \big( (N-\tfrac{\alpha}{2}) \log(1-N^{-1/3}) \big)
\geq 
  \exp( - N^{2/3} - \tfrac12 N^{1/3} - C_\alpha ) \,,
$$
where $C_\alpha$ denotes a constant depending only on $\alpha$
which need not be the same at~each occurrence. 
It follows that
\begin{align*}
      \left| \frac{\exp \left( -4N \frac{1-re^{i\varphi}}{1+re^{i\varphi}} \right)}{(Re^{it})^{(m+n)/2}} \right|
&\leq \exp \left( \frac{-4N^{2/3}}{2r(1+\cos\varphi) + (1-r)^2} + N^{2/3} + \tfrac12 N^{1/3} + C_\alpha \right) \\
&\leq \exp \left( \frac{-4N^{2/3} + 2r(1+\cos\varphi)N^{2/3} + r(1+\cos\varphi) N^{1/3} + C_\alpha}{2r(1+\cos\varphi) + (1-r)^2} \right) \,.
\end{align*}
(Here some uniformly bounded terms have been absorbed into the constant $C_\alpha$.)
Since $r = \sqrt{1-N^{-1/3}} \leq 1 - \tfrac12N^{-1/3}$, the numerator is bounded above by
$$
-4N^{2/3} + 2(1+\cos\varphi)N^{2/3} - (1+\cos\varphi) N^{1/3} + r(1+\cos\varphi) N^{1/3} + C_\alpha \leq C_\alpha \,,
$$
and \eqref{step22a} is proved.

Furthermore, the same arguments as those leading to \eqref{step12c} yield
\begin{align}
\label{step22b}
|f_2(Re^{it})| = \myo \left( |1-Re^{it}|^{1/2} /\, \sqrt{N} \right) \,.
\end{align}

Hence, by \eqref{step22a}, \eqref{step22b}, and symmetry, 
in order to complete the proof of \eqref{claim22},
it~remains to show that for $a \geq a_0(\delta)$,
\begin{align}
\label{step22c}
\frac{1}{2 \pi} \int_{a/N^{1/3}}^{\pi} \frac{1}{|1-Re^{it}|^{3/2}} \ dt \leq \delta N^{1/6} \,.
\end{align}
But this can be proved in the same way as \eqref{stdest}.
\end{proof}

With essentially the same proof, we obtain the following result:

\begin{theorem}
\label{real-A}
Let $f$ denote the second-order correlation function of a real sample covariance matrix
satisfying our standing moment conditions.
For any $\alpha \in \mathbb{N}$, $\mu,\nu \in \mathbb{R}$, setting
$$
Z_N(\mu,\nu) := \left( 16N^2 + 4 (\mu + \nu) N^{4/3} + \mu \nu N^{2/3} \right)^{\alpha/2} \, \exp \left( - 4N - \tfrac12(\mu+\nu) N^{1/3} \right) \,,
$$
we have
\begin{multline}
\label{real-AF}
\lim_{N \to \infty} \left( 4 \, Z_N(2^{4/3} \mu,2^{4/3} \nu) \cdot \frac{f(n,m;4N+2^{4/3} \mu N^{1/3},4N+2^{4/3} \nu N^{1/3})}{n! \, m!} \right) \\
= \exp(b^*) \, \widetilde{\mathbb{A}}(\mu,\nu) \,, 
\end{multline}
where $b^* := (b-3)$ and $\widetilde{\mathbb{A}}$ is defined as in \eqref{airy2}.
\end{theorem}

\bigskip

\section{Asymptotics at the Hard Edge of the Spectrum}

In order to zoom in at the hard edge of the spectrum,
we have to make the replacements
$$
\mu \mapsto \mu / 4N \quad \text{and} \quad \nu \mapsto \nu / 4N \,,
$$
where $\mu,\nu \in \mathbb{R}_{+}$. We then have the following result:

\begin{theorem}
\label{complex-J}
Let $f$ denote the second-order correlation function of a complex sample covariance matrix
satisfying our standing moment conditions.
For any $\alpha \in \mathbb{N}$, $\mu,\nu \in \mathbb{R}_{+}$, 
setting $Z_N(\mu,\nu) := (\mu\nu/N^2)^{\alpha/2}$, we have
\begin{multline}
\label{complex-JF}
\lim_{N \to \infty} \left( (1/4) \, N^{-1} \, Z_N(\mu/4,\nu/4) \cdot \frac{f(n,m;\mu/4N,\nu/4N)}{n! \, m!} \right) \\
= \exp(b^*) \, \mathbb{J}_\alpha(\mu,\nu) \,, 
\end{multline}
where $b^* := 2(b-\tfrac34)$ and $\mathbb{J}_\alpha$ is defined as in \eqref{bessel}.
\end{theorem}

\begin{proof}[Proof of Theorem \ref{complex-J}]
By (\ref{intrep2}), we have
\begin{multline}
\label{intrep31}
N^{-1} \, Z_N(\mu,\nu) \cdot \frac{f(n,m;\mu/N,\nu/N)}{n! \, m!} \\ = 
\frac{1}{2\pi i N} \int_\gamma 
\frac{\exp \left( - (\mu+\nu) N^{-1} \cdot \frac{z}{1-z} + b^*z \right) \cdot I_\alpha(w(z))}{(1-z)^2 \cdot z^{(n+m)/2}} 
\ \frac{dz}{z} \,,
\end{multline}
where now
$$
w(z) := 2 \sqrt{\mu\nu} N^{-1} \cdot \frac{\sqrt{z}}{1-z} \,.
$$
For the radius $R$ we make the same choice as for the bulk of the spectrum,
namely $R = 1 - 1/N$.
Similarly as in the preceding sections, we reduce the proof
to~the~following two claims:

On the one hand, we will show that for any (fixed) $a > 0$,
\begin{align}
\label{claim31}
\frac{1}{2\pi i N} \int_{\gamma'} \frac{\exp \left( - (\mu+\nu) N^{-1} \cdot \frac{z}{1-z} + b^*z \right) \cdot I_\alpha(w(z))}{(1-z)^2 \cdot z^{(n+m)/2}} \, \frac{dz}{z}
\overset{N \to \infty}{-\!\!\!-\!\!\!-\!\!\!\to}
J(a) \,,
\end{align}
where 
$$
J(a) := \frac{1}{2\pi} \int_{-a}^{+a} \frac{\exp \big( (1-iu) - (\mu+\nu)/(1-iu) \big) \, I_\alpha \big( 2\sqrt{\mu\nu}/(1-iu) \big)}{(1-iu)^2} \, du \,.
$$

On the other hand, we will show that for any $\delta > 0$,
there exists some $a_0(\delta) > 0$  such that for $a \geq a_0(\delta)$,
we have
\begin{align}
\label{claim32}
\left|
\frac{1}{2\pi i N} \int_{\gamma''} \frac{\exp \left( - (\mu+\nu) N^{-1} \cdot \frac{z}{1-z} + b^*z \right) \cdot I_\alpha(w(z))}{(1-z)^2 \cdot z^{(n+m)/2}} \, \frac{dz}{z}
\right|
\leq 
\delta
\end{align}
for all $N \in \mynat$ large enough.

Once these two claims are established, the proof of Theorem \ref{complex-J}
is completed using similar arguments as in the preceding section.
\eqref{claim31} and \eqref{claim32} imply that
\begin{align*}
\frac{1}{2\pi i N} \int_{\gamma} \frac{\exp \left( - (\mu+\nu) N^{-1} \cdot \frac{z}{1-z} + b^*z \right) \cdot I_\alpha(w(z))}{(1-z)^2 \cdot z^{(n+m)/2}} \, \frac{dz}{z}
\overset{N \to \infty}{-\!\!\!-\!\!\!-\!\!\!\to}
J(\infty) \,,
\end{align*}
where
$$
J(\infty) = \frac{\exp(b^*)}{2\pi i} \int_{1-i\infty}^{1+i\infty} \frac{\exp \big( z - (\mu+\nu)/z \big) \, I_\alpha \big( 2\sqrt{\mu\nu}/z \big)}{z^2} \, dz \,.
$$
Substituting $z = z/4$, $dz = dz/4$, 
and shifting the path back to the line $1 + i\mathbb{R}$
(which is easily justified by Cauchy's theorem), we obtain
$$
\frac{4\exp(b^*)}{2\pi i} \int_{1-i\infty}^{1+i\infty} \frac{\exp \big( \tfrac{1}{4} z - 4(\mu+\nu)/z \big) \, I_\alpha \big( 8\sqrt{\mu\nu}/z \big)}{z^2} \, dz \,.
$$
Making the replacements $\mu \mapsto \mu/4$, $\nu \mapsto \nu/4$
and dividing by $4$, we further obtain
$$
\frac{\exp(b^*)}{2\pi i} \int_{1-i\infty}^{1+i\infty} \frac{\exp \big( \tfrac{1}{4} z - (\mu+\nu)/z \big) \, I_\alpha \big( 2\sqrt{\mu\nu}/z \big)}{z^2} \, dz \,.
$$
The subsequent Lemma \ref{lemma} states that this is equal to $\exp(b^*) \, \mathbb{J}_{\alpha}(\mu,\nu)$,
thereby proving \eqref{complex-JF}.

\medskip

To prove (\ref{claim31}), we proceed by Taylor expansion once more.
A major difference to the previous situations is given by the fact that,
for $|u| \leq a$,
$$
  w(Re^{iu/N}) 
= 2 \sqrt{\mu\nu} N^{-1} \cdot \frac{\sqrt{Re^{iu/N}}}{1-Re^{iu/N}}
= 2 \sqrt{\mu\nu} \cdot \frac{1}{1-iu} \left( 1 + \myo_a(1/N) \right)
$$
remains bounded. Since, for $|u| \leq a$,
\begin{align*}
\frac{Re^{iu/N}}{1-Re^{iu/N}} &= \frac{N}{1-iu} \left( 1 + \myo_a(1/N) \right) \,,
\\
(1-Re^{iu/N})^2 &= (1-iu)^2 / N^2 \left( 1 + \myo_a(1/N) \right) \,,
\\
(Re^{iu/N})^{(n+m)/2} &= \exp(-(1-iu)) \left( 1 + \myo_a(1/N) \right) \,,
\end{align*}
we obtain
\begin{multline*}
\frac{1}{2\pi i N} \int_{\gamma'} \frac{\exp \left( - (\mu+\nu) N^{-1} \cdot \frac{z}{1-z} + b^*z \right) \cdot I_\alpha(w(z))}{(1-z)^2 \cdot z^{(n+m)/2}} \, \frac{dz}{z} \\
=
\frac{\exp(b^*)}{2\pi N^2} \int_{-a}^{+a} \frac{\exp \big( -(\mu+\nu)/(1-iu) \big) \, I_\alpha \big( 2\sqrt{\mu\nu}/(1-iu) \big)}{(1-iu)^2/N^2 \cdot \exp(-(1-iu))} \, \left( 1+o_a(1) \right) \, du \,,
\end{multline*}
whence \eqref{claim31}.

\medskip

\eqref{claim32} follows immediately from the estimates
$$
\left| \frac{z}{1-z} \right| \leq N \,,
\quad
\left| \frac{\sqrt{z}}{1-z} \right| \leq N \,,
\quad
\left| I_\alpha(z) \right| \leq \exp (\re z) \,,
\quad
\left| z^{(n+m)/2} \right| = \Omega(1) \,,
$$
and, for $a \geq a_0(\delta)$,
$$
\left| \int_{a/N}^{\pi} \frac{1}{|1-Re^{it}|^2} \, dt \right| \leq \delta N \,.
$$
The last estimate is proved similarly to \eqref{stdest}.
\end{proof}

\begin{lemma}
\label{lemma}
For any $\alpha \in \mathbb{N}$, $\mu,\nu \in \mathbb{R}_{+}$,
\begin{multline*}
\frac{1}{2\pi i} \int_{1-i\infty}^{1+i\infty} \frac{\exp \left(\tfrac{1}{4} z - (\mu+\nu)/z \right) \, I_\alpha \left( 2\sqrt{\mu\nu}/z \right)}{z^2} \, dz \\ = \frac{J_\alpha(\sqrt\mu) \, \sqrt\nu \, J_\alpha'(\sqrt\nu) - \sqrt\mu \, J_\alpha'(\sqrt\mu) \, J_\alpha(\sqrt\nu)}{2(\mu-\nu)} \,.
\end{multline*}
\end{lemma}

\begin{proof}
We start from the following integral representation
for the product of two Bessel functions (see \eg p.~281 in \cite{BMP}):
$$
\frac{1}{2\pi i} \int_{1-i\infty}^{1+i\infty} e^{tz} \, \frac{\exp(-(x^2+y^2)/z) I_\alpha(2xy/z)}{z} \, dz = J_\alpha(2x\sqrt{t}) J_\alpha(2y\sqrt{t}) \,.
$$
Differentiating with respect to $x$ and $y$, we obtain
\begin{multline*}
\frac{1}{2\pi i} \int_{1-i\infty}^{1+i\infty} e^{tz} \, \frac{\exp(-(x^2+y^2)/z) \left( -2x I_\alpha(2xy/z) + 2y I_\alpha'(2xy/z) \right)}{z^2} \, dz \\ = 2\sqrt{t} J_\alpha'(2x\sqrt{t}) J_\alpha(2y\sqrt{t})
\end{multline*}
and
\begin{multline*}
\frac{1}{2\pi i} \int_{1-i\infty}^{1+i\infty} e^{tz} \, \frac{\exp(-(x^2+y^2)/z) \left( -2y I_\alpha(2xy/z) + 2x I_\alpha'(2xy/z) \right)}{z^2} \, dz \\ = 2\sqrt{t} J_\alpha(2x\sqrt{t}) J_\alpha'(2y\sqrt{t})
\end{multline*}
respectively.
Multiplying the former equation by $x$ and the latter equation by $y$,
taking the difference and dividing by $2(x^2-y^2)$ yields
\begin{multline*}
\frac{1}{2\pi i} \int_{1-i\infty}^{1+i\infty} e^{tz} \, \frac{\exp(-(x^2+y^2)/z) I_\alpha(2xy/z)}{z^2} \, dz \, \\ = \frac{J_\alpha(2x\sqrt{t}) \, 2y\sqrt{t} \, J_\alpha'(2y\sqrt{t}) - 2x\sqrt{t} \, J_\alpha'(2x\sqrt{t}) \, J_\alpha(2y\sqrt{t})}{2(x^2-y^2)} \,.
\end{multline*}
The assertion now follows by letting $x := \sqrt{\mu}$, $y := \sqrt{\nu}$, 
and $t := 1/4$.
\end{proof}

The following result can be proved completely analogously to Theorem~\ref{complex-J}:

\begin{theorem}
\label{real-J}
Let $f$ denote the second-order correlation function of a real sample covariance matrix
satisfying our standing moment conditions.
For any $\alpha \in \mathbb{N}$, $\mu,\nu \in \mathbb{R}_{+}$, 
setting $Z_N(\mu,\nu) := (\mu\nu/N^2)^{\alpha/2}$, we have
\begin{multline}
\label{real-JF}
\lim_{N \to \infty} \left( (1/16) \, N^{-2} \, Z_N(\mu/4,\nu/4) \cdot \frac{f(n,m;\mu/4N,\nu/4N)}{n! \, m!} \right) \\
= \exp(b^*) \, \widetilde{\mathbb{J}}_\alpha(\mu,\nu) \,, 
\end{multline}
where $b^* := (b-3)$ and $\widetilde{\mathbb{J}}_\alpha$ is defined as in \eqref{bessel2}.
\end{theorem}

\bigskip

\end{document}